\documentclass[12pt]{amsart}
\usepackage{amsmath,amssymb,amsbsy,amsfonts,latexsym,amsopn,amstext,cite,
                                               amsxtra,euscript,amscd,bm,mathabx,mathrsfs, abraces}
\usepackage{url}
\usepackage[colorlinks,linkcolor=blue,anchorcolor=blue,citecolor=blue,backref=page]{hyperref}
\usepackage{color}
\usepackage{graphics,epsfig}
\usepackage{graphicx}
\usepackage{float} 
\usepackage[english]{babel}
\usepackage{mathtools}
\usepackage{todonotes}
\usepackage{url}
\usepackage[colorlinks,linkcolor=blue,anchorcolor=blue,citecolor=blue,backref=page]{hyperref}

\usepackage[norefs,nocites]{refcheck}

\hypersetup{breaklinks=true}

\usepackage[norefs,nocites]{refcheck}
\usepackage[english]{babel}
\begin{document}

\newtheorem{thm}{Theorem}
\newtheorem{lem}[thm]{Lemma}
\newtheorem{claim}[thm]{Claim}
\newtheorem{cor}[thm]{Corollary}
\newtheorem{prop}[thm]{Proposition} 
\newtheorem{definition}[thm]{Definition}
\newtheorem{question}[thm]{Open Question}
\newtheorem{conj}[thm]{Conjecture}
\newtheorem{rem}[thm]{Remark}
\newtheorem{prob}{Problem}

\def\ccr#1{\textcolor{red}{#1}}
\def\cco#1{\textcolor{orange}{#1}}
\def\ccg#1{\textcolor{cyan}{#1}}

\newtheorem{ass}[thm]{Assumption}

\newtheorem{lemma}[thm]{Lemma}

\newcommand{\GL}{\operatorname{GL}}
\newcommand{\SL}{\operatorname{SL}}
\newcommand{\lcm}{\operatorname{lcm}}
\newcommand{\ord}{\operatorname{ord}}
\newcommand{\Tr}{\operatorname{Tr}}
\newcommand{\Span}{\operatorname{Span}}

\numberwithin{equation}{section}
\numberwithin{thm}{section}
\numberwithin{table}{section}

\def\vol {{\mathrm{vol\,}}}
\def\squareforqed{\hbox{\rlap{$\sqcap$}$\sqcup$}}
\def\qed{\ifmmode\squareforqed\else{\unskip\nobreak\hfil
\penalty50\hskip1em\null\nobreak\hfil\squareforqed
\parfillskip=0pt\finalhyphendemerits=0\endgraf}\fi}

\def \balpha{\bm{\alpha}}
\def \bbeta{\bm{\beta}}
\def \bgamma{\bm{\gamma}}
\def \blambda{\bm{\lambda}}
\def \bchi{\bm{\chi}}
\def \bphi{\bm{\varphi}}
\def \bpsi{\bm{\psi}}
\def \bomega{\bm{\omega}}
\def \btheta{\bm{\vartheta}}
\def \bmu{\bm{\mu}}
\def \bnu{\bm{\nu}}

\newcommand{\bfxi}{{\boldsymbol{\xi}}}
\newcommand{\bfrho}{{\boldsymbol{\rho}}}

\def\vX{\mathbf X}
\def\vY{\mathbf Y}

\def\cA{{\mathcal A}}
\def\cB{{\mathcal B}}
\def\cC{{\mathcal C}}
\def\cD{{\mathcal D}}
\def\cE{{\mathcal E}}
\def\cF{{\mathcal F}}
\def\cG{{\mathcal G}}
\def\cH{{\mathcal H}}
\def\cI{{\mathcal I}}
\def\cJ{{\mathcal J}}
\def\cK{{\mathcal K}}
\def\cL{{\mathcal L}}
\def\cM{{\mathcal M}}
\def\cN{{\mathcal N}}
\def\cO{{\mathcal O}}
\def\cP{{\mathcal P}}
\def\cQ{{\mathcal Q}}
\def\cR{{\mathcal R}}
\def\cS{{\mathcal S}}
\def\cT{{\mathcal T}}
\def\cU{{\mathcal U}}
\def\cV{{\mathcal V}}
\def\cW{{\mathcal W}}
\def\cX{{\mathcal X}}
\def\cY{{\mathcal Y}}
\def\cZ{{\mathcal Z}}
\def\Ker{{\mathrm{Ker}}}
\def\diag{{\mathrm{diag}}}

\def\sA{{\mathscr A}}

\def\NmQR{N(m;Q,R)}
\def\VmQR{\cV(m;Q,R)}

\def\Xm{\cX_m}

\def \A {{\mathbb A}}
\def \B {{\mathbb A}}
\def \C {{\mathbb C}}
\def \F {{\mathbb F}}
\def \G {{\mathbb G}}
\def \L {{\mathbb L}}
\def \K {{\mathbb K}}
\def \Q {{\mathbb Q}}
\def \R {{\mathbb R}}
\def \Z {{\mathbb Z}}

\def \fA{\mathfrak A}
\def \fC{\mathfrak C}
\def \fL{\mathfrak L}
\def \fR{\mathfrak R}
\def \fS{\mathfrak S}

\def \fUg{{\mathfrak U}_{\mathrm{good}}}
\def \fUm{{\mathfrak U}_{\mathrm{med}}}
\def \fV{{\mathfrak V}}
\def \fG{\mathfrak G}
\def \f{\mathfrak G}

\def\e{{\mathbf{\,e}}}
\def\ep{{\mathbf{\,e}}_p}
\def\eq{{\mathbf{\,e}}_q}

 \def\\{\cr}
\def\({\left(}
\def\){\right)}
\def\fl#1{\left\lfloor#1\right\rfloor}
\def\rf#1{\left\lceil#1\right\rceil}

\def\Im{{\mathrm{Im}}}

\def \ovF {\overline F}

\newcommand{\pfrac}[2]{{\left(\frac{#1}{#2}\right)}}

\def \Prob{{\mathrm {}}}
\newcommand{\disc}{\operatorname{discr}}
\def\e{\mathbf{e}}
\def\ep{{\mathbf{\,e}}_p}
\def\epp{{\mathbf{\,e}}_{p^2}}
\def\em{{\mathbf{\,e}}_m}

\def\Res{\mathrm{Res}}
\def\Orb{\mathrm{Orb}}

\def\vec#1{\mathbf{#1}}
\def \va{\vec{a}}
\def \vb{\vec{b}}
\def \vc{\vec{c}}
\def \vs{\vec{s}}
\def \vu{\vec{u}}
\def \vv{\vec{v}}
\def \vw{\vec{w}}
\def\vlam{\vec{\lambda}}
\def\flp#1{{\left\langle#1\right\rangle}_p}

\def\mand{\qquad\mbox{and}\qquad}

\title[Integer matrices with a given discriminant]
{On the sparsity of non-diagonalisable integer matrices  and matrices with a given discriminant}

\author[A. Ostafe] {Alina Ostafe}
\address{School of Mathematics and Statistics, University of New South Wales, Sydney NSW 2052, Australia}
\email{alina.ostafe@unsw.edu.au}

\author[I. E. Shparlinski] {Igor E. Shparlinski}
\address{School of Mathematics and Statistics, University of New South Wales, Sydney NSW 2052, Australia}
\email{igor.shparlinski@unsw.edu.au}

\begin{abstract}  We consider the set  $\cM_n\(\Z; H\)$ of $n\times n$-matrices with 
integer elements of size at most $H$ and obtain upper  bounds on the number of  matrices  
 from $\cM_n\(\Z; H\)$, for which the characteristic polynomial has a fixed discriminant $d$. When $d=0$, this corresponds to counting matrices with a repeated eigenvalue, and thus is related to counting non-diagonalisable matrices. For $d\ne 0$, this problem seems not to have been studied previously, while for $d=0$, both our approach and the final result improve
 on those of   A.~J.~Hetzel, J.~S.~Liew and K.~Morrison  (2007).    
 \end{abstract}



\maketitle

\tableofcontents

\section{Introduction}
 \subsection{Set-up and motivation}
For positive integers $n$ and   $H$, we let $\cM_n\(\Z;H\)$  denote the
set of all  $n\times n$ matrices  
\[
A = \(a_{ij}\)_{1 \le i,j \le n} 
\]  
with integer entries of size $|a_{ij}| \le H$.  In particular, $\cM_n\(\Z; H\)$ is 
of cardinality $\# \cM_n\(\Z; H\) = \(2H +1\)^{n^2}$. 

Various  problems of  {\it arithmetic statistics\/} with matrices  from the set  $\cM_n\(\Z; H\)$  have received a lot of attention
see, for example,~\cite{ALPS, Afif, BOS, E-BLS, HOS} for some 
recent results and further references. Here we consider a rather natural 
question of counting the number $N_n(H)$ of 
matrices  from $\cM_n\(\Z; H\)$ which are not diagonalisable.

We note that previously this question has been studied by   Hetzel, Liew and Morrison~\cite{HLM},
where in fact the number $R_n(H)$ of 
matrices  from $\cM_n\(\Z; H\)$ with a repeated eigenvalue is estimated (observe that we trivially have 
$R_n(H) \ge N_n(H)$).
Note that~\cite[Theorem~2.1]{HLM} merely asserts $R_n(H)  = o(H^{n^2})$ for 
a fixed $n$ and $H\to\infty$. However, analysing the proof of~\cite[Theorem~2.1]{HLM}, 
exploiting  the vanishing of the discriminant of the characteristic polynomial of 
such matrices, one easily sees that  it actually leads to the bound
\begin{equation}
\label{eq:HLM-Bound}
R_n(H)  = O\(H^{n^2-1}\), 
\end{equation}
where hereafter the implied constants depend only on $n$. 

As noted in Remark~\ref{rem:conditional} below, 
conditionally on~\cite[Conjecture~1.1]{HOS} on the upper bound on the number of matrices with a 
give characteristic polynomial, see~\eqref{eq:conj Pn} below,  it is natural to expect that
\begin{equation}
\label{eq:conj Rn}
R_n(H)\le H^{n^2 -2n  +2+o(1)}.
\end{equation} 

One can also interpret $R_n(H)$ as counting matrices in $ \cM_n(\Z; H)$ for which the characteristic polynomial has discriminant equal to zero. We adapt this point of view and count more generally matrices in $ \cM_n(\Z; H)$ for which the characteristic polynomial has a given discriminant. 

For $A\in \cM_n(\Z; H)$, let $f_A\in\Z[X]$ be its characteristic polynomial and we define the {\it discriminant of $A$}, denoted $\disc(A)$, as the discriminant of $f_A$. For $d\in \Z$  fixed, we define
\[
\cR_{n}(d,H)=\{A\in  \cM_n(\Z; H) :~\disc(A)=d\}
\]
and $R_{n}(d,H)=\#\cR_{n}(d,H)$. If $d=0$, then $R_n(H)=R_{n}(0, H)$.

In this paper we present upper bounds for $R_{n}(d,H)$, significantly improving~\eqref{eq:HLM-Bound} for $d=0$. 
It appears that for $d\ne0$  this problem has not been studied before. In this case,  our bound is weaker, however
stronger than the one of the strength of~\eqref{eq:HLM-Bound} (note that the approach of~\cite{HLM}
 can be adapted  to $d \ne 0$ without any problems).

  \subsection{Main results}

We recall that   the notations $U\ll V$  and $V \gg U$ are both equivalent to the statement $|U|\leq c V$, 
 for some  constant $c> 0$, which throughout this work may depend only on  $n$. 
 
 Our first main result gives  a significant improvement of~\eqref{eq:HLM-Bound}. 
 
\begin{thm}
\label{thm:MulEigenval}
We have 
\[
H^{(n^2+n-2)/2}\ll N_n(H) \le R_n(H) \ll H^{n^2 -\Delta_n} \log H,
\]
where 
\[
\Delta_n = \max_{r=1, \ldots, \fl{n/2}} \min\left\{n - \frac{r(r+1)}{2}, r+1 \right\}, 
\]
while  for $n=2,5,7,8$, we also have 
\begin{align*} 
& R_2(H) \le H^{2+o(1)},   &R_5(H) \ll H^{5^2-4}\log H,\\
& R_7(H)\ll H^{7^2-4}\log H,    & R_8(H)\ll H^{8^2-5}\log H.
\end{align*} 
\end{thm}

We notice that for other small values of $n$ (for which the 
special treatment of Theorem~\ref{thm:MulEigenval} does not apply) 
we have
\[
\Delta_3 = \Delta_4 = 2,\qquad \Delta_6= \Delta_9=3, \qquad \Delta_{10} = 4,
\]
while for large $n$ we have
\[
\lim_{n\to \infty} \Delta_n/\sqrt{2n}  = 1.
\]

We also note that for $n=2$, the lower and upper bounds in Theorem~\ref{thm:MulEigenval} agree up to $H^{o(1)}$, that is, one has
\[
H^{2}\ll N_2(H)\le R_2(H) \le H^{2+o(1)},
\]
confirming~\eqref{eq:conj Rn}. 

The proof of  Theorem~\ref{thm:MulEigenval} uses a different approach than in~\cite{HLM}, and is based on studying possible factorisations
of the characteristic polynomials of matrices and then either using a generalisation of a result from~\cite{Shp} or
using modular reduction modulo an 
appropriately chosen prime $p$. After this we can use known results about the distribution of characteristic polynomials of matrices over finite fields~\cite{Rei}. 

The above approach, which  we use for $d=0$ does not apply for $d\ne 0$. In this case we study the absolute irreducibility 
of the corresponding variety and use the bound of 
Browning, Heath-Brown and Salberger~\cite[Corollary~2]{BrHBSa}
(see also~\cite{Salb}) 
to prove the following estimate. 

\begin{thm}
\label{thm:Discr}
Uniformly over  $d \ne 0$, as $H \to \infty$,  we have 
\[
R_{n}(d,H) \ll   H^{n^2 -  2+o(1)}. 
\] 
\end{thm}

We note that one of our auxiliary results, an upper bound of the number 
of matrices with an eigenvalue of degree at most $r$ over $\Q$, see Lemma~\ref{lem:low deg eigenval}, 
 improves a recent result of  Anderson and O'Dorney~\cite{AnOD} in some ranges of $r$, 
 which are important for our argument. 
We also discuss  some other applications of our technical tools, in particular improving a 
result of Assing, Blomer and Nelson~\cite[Corollary~1.7]{AsBlNe}  
 in some ranges of parameters, see Section~\ref{sec:Appl}. 

  \subsection{Some comments and open questions}
  
We note that our approach does not seem to work for $\SL_n(\Z)$ matrices. 
However,  an upper bound 
\begin{equation}
\label{eq:bound Ln}
\#\cL_n\(d,H\)  \ll H^{n^2-n - 1/\(4\rf{(n-1)/2}\(n^2+n\)\)}
\end{equation} 
on the cardinality of 
\[
\cL_n\(d,H\)  = \{A\in \cM_n\(\Z; H\) \cap \SL_n(\Z)  :~\disc(A)=d\}
\]
is  essentially implicit in~\cite[Example~1.7]{GN} (at least for $d=0$) and
is derived in~\cite[Section~3]{GN} from~\cite[Theorem~3.1]{GN}. 
We recall that by~\cite[Example~1.6]{DRS} we have
\[
H^{n^2-n} \ll \# \{A\in \cM_n\(\Z; H\) \cap \SL_n(\Z)\} \ll H^{n^2-n}.
\]
The only missing ingredient in~\cite{GN}  is a justification of the fact the affine variety 
defined by 
\[
\det \vX = 1, \quad \disc(\vX) = d, \qquad \qquad \vX \in \C^{n\times n},
\]
is of dimension $n^2-2$.  To see that we use the multilinear equation $\det \vX = 1$ 
to express $x_{11}$ 
via the remaining $n^2-1$ variables and substitute this in $ \disc(\vX) = d$. This leads to
a polynomial equation in $n^2-1$ variables.
This equation is non-trivial since 
otherwise all $ \SL_n(\C)$ matrices have the same discriminant $d$, which is false.
Indeed,  the identity matrix has a vanishing discriminant, while the 
diagonal matrix with roots of unity $\exp(2\pi i  \nu/n)$, $\nu =1, \ldots, n$, on the 
main diagonal has a non-zero discriminant. 

It is certainly very interesting to investigate whether the methods 
of~\cite{DRS, GNY} apply to estimating $R_n(H)$.

Similarly to our conjectured bound~\eqref{eq:conj Rn}, one can expect 
that~\eqref{eq:bound Ln} can be improved. However before making any conjectures
we need to understand the number of monic polynomials  $f$ of degree $n$ over $\Z$, 
which are similar to
those in  Remark~\ref{rem:conditional} below, but have the constant coefficient
$f(0) = 1$.  

Another interesting direction, is obtaining analogues of our results for some 
other special classes of matrices, such as symmetric matrices.  It seems plausible 
that our techniques apply to this case as well.

\section{Preliminaries on  characteristic polynomial, determinants and discrimiants} 

  \subsection{Counting matrices with a given characteristic polynomial}

For a given polynomial $f \in \Z[X]$ let $P_n(H;f)$ denote the number 
of matrices  $A\in \cM_n\(\Z; H\)$ with a   characteristic polynomial $f$.

Some upper bounds on $P_n(H;f)$ can be found in~\cite{HOS}. 
Here we give a short 
proof in the case of $n=2$, which we need here.

\begin{lemma}
\label{lem:char poly}
Uniformly over polynomials  $f\in \Z[X]$ of degree $\deg f = 2$ we have 
\[
P_{2} (H;f) \le  H^{1+o(1)}, \qquad  H \to \infty.
\] 
\end{lemma} 

\begin{proof} 
Let 
\[ 
A  = \begin{pmatrix}a & b\\c & d \end{pmatrix}.
\]
If the characteristic polynomial of $A$ is $f(X) = X^2 - u X + v$, this fixes the trace and the determinant of $A$.
Thus we have 
\[
a+d = u \mand ad - bc = v.
\]
Therefore $bc = a(u-a) -v$, and when the triple $(a,b,c)$ is fixed, $d$ is uniquely defined. 

There are at most $2$ values of $a$ with $a(u-a) - v= 0$, in which case either $b=0$ 
of $c=0$ (or both). So the are at most $4H$ such matrices.

Otherwise for each $a$, the classical bound on the 
divisor function,   see~\cite[Equation~(1.81)]{IwKow}, implies that 
there are at most $H^{o(1)}$ choices for $(b,c)$. 
 So the are at most $H^{1+o(1)}$ such matrices.
\end{proof} 

\begin{rem}
\label{rem:conditional}
It has been conjectured in~\cite[Conjecture~1.1]{HOS} that 
\begin{equation}
\label{eq:conj Pn}
P_{n}(H;f) \le H^{n(n-1)/2+o(1)},
\end{equation}
as $H \to \infty$, uniformly with respect to $f$. 
Therefore, since $R_n(H)$ counts matrices in $\cM_n(\Z;H)$ with a multiple eigenvalue,
 it is natural to expect that matrices with only one 
double eigenvalue dominate the count, and thus their characteristic 
polynomial is of the form $f(X) = (X-a)^2 g(X)$ where
\[ a \in \Z \quad \text{and} \quad
g(X) =  X^{n-2} + b_{1} X^{n-3} + \dots + b_1 X + b_{n-2} \in \Z[X],
\]
with $a \ll H$ and $b_i \ll   H^i$, $i=1, \ldots, n-2$. Therefore,  
there are $O\(H^{1+(n-1)(n-2)/2}\)$ such polynomials $f$, 
and applying the bound~\eqref{eq:conj Pn}, we obtain
\[
R_n(H)\le H^{1+(n-1)(n-2)/2} \cdot H^{n(n-1)/2+o(1)}=H^{n^2 -2n  +2+o(1)}.
\] 
\end{rem}

For a prime $p$, 
we also recall a result  of Reiner~\cite[Theorem~2]{Rei} that gives an explicit formula on the number $P_{n,p}(f) $ of matrices with a given characteristic polynomial over the field $\F_p$ of $p$ elements.
However, for our purposes, we only need the following asymptotical formula for $P_{n,p}(f)$.

\begin{lemma}\label{lem:Rei}  For any monic polynomial $f\in \F_p[X]$ of degree $n$ we have 
\[
P_{n,p}(f) = \(1 + o(1)\) p^{n^2-n},  \qquad \text{as}\ p \to \infty.
\]
\end{lemma}

\subsection{Determinants of shifted matrices}
   
The following result  can be obtained
by a slight extension of  the argument in the 
proof of~\cite[Theorem~4]{Shp}, which is given for an integral $n\times n$ matrix $K$.
However, the proof, which is based on eliminating $K$, never appeals to its integrality.  
In particular, we have. 

\begin{lemma}
\label{lem:shift sing matr}
Uniformly over  $n\times n$ matrices $K \in \C^{n\times n}$  and complex $\alpha \in \C$, for 
\[
J_n(K; H, \alpha)= \#\left \{A\in \cM_n\(\Z; H\):~\det(A-K) = \alpha \right \}  
\]
we have
\[
J_n(K; H, \alpha)  \ll H^{n^2-n} \log H. 
\]
\end{lemma} 

\begin{proof} Assume $A,B \in\cM_n\(\Z; H\)$ are obtained from 
 an $n\times (n-1)$-matrix
$R$ by  augmenting it by two vectors $\vec{a}, \vec{b} \in \Z^n$, 
respectively. That is,
\[
A = \(R\vert \vec{a}\) \qquad \text{and}\qquad 
B = \(R\vert \vec{b}\).
\]
If $\det (A -K)= \det (B-K)$ then putting 
\[
 \vec{c}= \vec{a}-\vec{b} \qquad \text{and}\qquad 
C = \(R\vert   \vec{c}\),
\]
we deduce that $\det (C-K_1)  =0$, where the matrix $K_1$ is obtained from $K$ by replacing its $n$th column
by a  zero vector. Indeed, to see  this,  it is
enough to expand $C$ with respect to the last column.

Let 
\[
K = \(\kappa_{ij} \)_{i,j=1}^n .
\]
Therefore, for any $R$ and $\alpha$, we have
\begin{equation}
\label{eq:Reduction 1}
\begin{split}
 \# \{\vec{a} = (a_1&, \ldots, a_n) \in \Z^n:  \\
 &   \det\( \(R\vert \vec{a}\)-K\)= \alpha, \
|a_i| < H, \ i =1, \ldots, n\}\\
&\qquad \le  \#  \{\vec{c} =  (c_1, \ldots,c_n) \in \Z^n: \\
& \qquad \qquad   \det \(C - K_1\)=0, \
 |c_i| < 2H, \ i =1, \ldots, n\}.
\end{split}
\end{equation}
Indeed, if $\vec{a}_1, \ldots, \vec{a}_J$ is the list of 
elements from the first set in~\eqref{eq:Reduction 1} then
the vectors $\vec{c}_j = \vec{a}_1- \vec{a}_j$, $j=1, \ldots, J$, 
are distinct and belong to the second set in~\eqref{eq:Reduction 1}.

Summing~\eqref{eq:Reduction 1}
over all  integral  matrices 
$R = (r_{ij})_{i,j=1}^{n,n-1}$ such 
that
\[
|r_{ij}| \le H, \qquad 
1\le i \le n, \ 1\le j \le n-1,
\]
and observing that   $C    \in\cM_n\(\Z; 2H\)$, 
we obtain
\begin{equation}
\label{eq:Reduction 2}
J_n(K; H, \alpha) \le  J_n(K_1; 2H, 0).
\end{equation}

Repeating the same argument with respect to the $(n-1)$th column 
of the matrix $K_1$, we obtain from~\eqref{eq:Reduction 2}
\[
J_n(K; H, \alpha) \le J_{n}(K_2;4H,0), 
\]
where now $K_2$ is obtained from $K$ by replacing its $(n-1)$th and $n$th columns
by a  zero vector. 

Continuing the same procedure, after $n$ steps we arrive to the
inequality
\begin{equation}
\label{eq:Reduction n}
J_n(K; H, \alpha) \le  J_n(K_n; 2^nH, 0)
\end{equation} 
where $K_n = O_n$ is the zero matrix.  By a result of  Katznelson~\cite[Theorem~1]{Kat} we have 
\[
J_n(K_n; 2^nH, 0)   \ll H^{n^2-n} \log H, 
\]
which together with~\eqref{eq:Reduction n} concludes the proof. 
\end{proof}

\subsection{Matrices with eigenalues of low degree}
Let $R_{n,r}(H)$ be the number  of  matrices $\cM_n\(\Z; H\)$ 
whose characteristic polynomial has  an irreducible factor of degree $r \le  n/2$.  
Anderson and  O'Dorney~\cite[Theorem~1.4]{AnOD} have shown that 
\[
R_{n,r}(H) \ll H^{n^2 - (n-r)/(n+r-1)}.
\]  
However for our purpose we need the following bound which is better for small values of $r$. 

\begin{lemma}
\label{lem:low deg eigenval} For any positive integer $r\le n/2$ we have 
\[
R_{n,r}(H) \ll  H^{n^2-n+r(r+1)/2} \log H. 
\]
\end{lemma} 

\begin{proof}
By the famous result of Gerschgorin~\cite[Satz~II]{Gersch},  all eigenvalues of a matrix  $A\in \cM_n\(\Z; H\)$ are of 
size $O(H)$. 
Hence, if $\lambda$ is an eigenvalue of $A\in \cM_n\(\Z; H\)$ of degree $d = [\Q(\lambda):\Q] $
then, since its conjugates are of  order $O(H)$, the minimal polynomial $F$ of $\lambda$ over $\Q$,
\[
F(X) = X^d + c_{1} X^{d-1} + \ldots +  c_{d-1} X + c_d \in \Z[X]
\]
has coefficients satisfying 
\[
c_{i} = O\(H^i\), \qquad i =1, \ldots, d.
\]
Thus  there are  at most $O\(H^{d(d+1)/2}\)$ possibilities for $F$. 

We now fix some positive  integer $r\le n/2$. We see that the set $\cL$ 
of all possible eigenvalues of degree  $[\Q(\lambda):\Q] \le r $ of matrices $A\in \cM_n\(\Z; H\)$ is of size
\begin{equation}
\label{eq:Set L}
\# \cL = O\(H^{r(r+1)/2}\). 
\end{equation}
For each $\lambda \in \cL$, since 
\[
\det(A-\lambda I_n ) = 0, 
\]
where $I_n$ is the $n\times n$ identity matrix,    we conclude by Lemma~\ref{lem:shift sing matr}
that there are $O\(H^{n^2-n} \log H\)$ possible matrices $A$. 

Hence, recalling~\eqref{eq:Set L}, we see that the total number of matrices $A\in \cM_n\(\Z; H\)$, which have an eigenvalue 
$\lambda$  of degree $[\Q(\lambda):\Q] \le r $ over $\Q$ is 
\[
R_{n,r}(H)  \ll \# \cL  \cdot H^{n^2-n} \log H\ll H^{n^2-n+r(r+1)/2} \log H, 
\]
which concludes the proof. 
\end{proof}

\subsection{Absolute irreducibility of the discriminant of a generic matrix} 
We treat 
\[
\disc(\vX) \in \Z\left[\vX\right], \qquad \vX = \(X_{ij}\)_{1 \le i,j \le n}, 
\]
 as a polynomial over $\Z$ in $n^2$ variables. 

\begin{lemma}
\label{lem:abs irred discr}
The polynomial $\disc(\vX)$ is absolutely irreducible over $\Q$.
\end{lemma} 

\begin{proof} 
Assume that  $\disc(\vX)$ can be factored  over the algebraic closure of $\Q$ as 
\begin{equation}
\label{eq:Fact}
\disc(\vX)= \prod_{\nu=1}^s F_\nu(\vX)^{k_\nu} ,
\end{equation}
into $s$ distinct absolutely irreducible polynomials 
$F_1, \ldots, F_s \in K[\vX]$, defined over some number field $K$, with
some multiplicities $k_1, \ldots, k_s \ge 1$.  Our goal is to show first that $s=1$
and then that $k_1 = 1$. 

We first note that for any integer $m \ge 1$, there exists an   integer $a$ 
such that $\gcd\(a(2a+1), m\) = 1$. In fact $a=m-1$ satisfies this condition.

We now use the above observation with $m = n(n-1)/2$ and  fix    some 
integer $a$ with $\gcd\(a(2a+1), n(n-1)/2\) = 1$.
By the Chebotarev Density Theorem, there are infinitely many primes in the progression 
$p \equiv 2a+1 \pmod {n(n-1)}$ which split completely 
in the field $K$.  

We also choose $p$ to be large enough so that by the   Ostrowski's theorem (see~\cite[Corollary~2B]{Schmidt}), 
the factorisation~\eqref{eq:Fact} induces the factorisation 
\begin{equation}
\label{eq:Fact p}
\disc(\vX)  \equiv  \prod_{\nu=1}^s \ovF_\nu(\vX)^{k_\nu}  \pmod \cP , \qquad	\ovF_1, \ldots, \ovF_s \in \F_p[\vX],  
\end{equation} 
in distinct absolutely irreducible polynomials for any prime ideal $\cP$ in $K$ lying over $p$. 

We now observe that the factorisation~\eqref{eq:Fact p}, together with the celebrated result 
of Lang and Weil~\cite[ Theorem~1]{LaWe}, which we use in the form
given by~\cite[Theorem~7.5]{CafMat}, implies that 
\begin{equation}
\label{eq:Count 1}
\# \{A \in \F_p^{n\times n}:~ \disc(A)=0\} = s p^{n^2-1} + O\(p^{n^2-3/2}\).
\end{equation}
On the other hand, by Lemma~\ref{lem:Rei} we have 
\begin{equation}
\label{eq: A->f}
\# \{A \in \F_p^{n\times n}:~  \disc(A) =0 \} =(1+o(1))  p^{n^2-n} N_0 , 
\end{equation}
where $N_0$ is the number of monic non-squarefree polynomials 
in  $ \cP_n(\F_p)$ and thus by the classical result of Carlitz~\cite[Section~6]{Carl}, for any $n \ge 2$, we have 
\begin{equation}
\label{eq: N0}
N_0 = p^{n-1}.
\end{equation}
Combining~\eqref{eq:Count 1},  \eqref{eq: A->f} and~\eqref{eq: N0} and taking $p$ to be sufficiently large
we see that $s=1$. 
 Hence, we can write~\eqref{eq:Fact} as
\[
\disc(\vX) =  F(\vX)^{k} ,
\]
with an absolutely irreducible polynomial $F\in K[\vX]$.

Finally, we want to conclude that $k=1$, that is, $\disc(\vX)$ is absolutely irreducible. It is enough to show that $\disc(\vX)$ cannot be a power of a polynomial by considering some specialisation of the variables $\vX$.

In particular, we consider the matrix 
$$
\vX^*=\diag(J,  2, \ldots, n-1),
$$
where 
\[
J=\(\begin{array}{ll} 0 & X_{1,2}\\ X_{2,1}&0\end{array}\).
\]
From the  Laplace expansion with respect to the first two rows, 
we see that the characteristic polynomial of $\vX^*$ is
\[
f(Y)=\det (\vX^* - Y\cdot I_n) =  g(Y) h(Y) \in  R[Y]
\]
where $R = K[X_{1,2}, X_{2,1} ]$ and 
\begin{align*}
  & g(Y) = \det (J - Y\cdot I_2) = Y^2 -   X_{1,2} X_{2,1} ,\\
  & h(Y) = \det \diag(2-Y, \ldots, n-1-Y) = \prod_{i=2}^{n-1} \(i-Y\).
\end{align*} 
The discriminants of $g$ and $h$ are
\[
\disc(g)= 4 X_{1,2} X_{2,1} \mand \disc(h) = \prod_{2 \le i < j \le n-1} (i-j)^2,
\]
and the resultant of $g$ and $h$ is given by
$$
\Res(g,h) =  \prod_{i=2}^{n-1} \(i^2 - X_{1,2} X_{2,1}\).
$$
Recall that 
$$\disc(f)=\disc(g)\disc(h)\Res(g,h)^2. $$ 
Since $\disc(g)$ and $\Res(g,h)$ are relatively prime polynomials in $R$, and $\disc(g)$ is squarefree while $\disc(h)$ is a scalar in $R$, we deduce that $\disc(f)$ cannot be a power of a polynomial in $R$, which concludes the proof.
\end{proof}

  \section{Proof of Theorem~\ref{thm:MulEigenval}}

\subsection{The general upper bound}
\label{sec:upper} 
The idea of the proof is, for arbitrary $1\le r\le n/2$, to bound $$R_n(H)\le R_{n,r}(H)+Q_{n,r}(H),$$
where $R_{n,r}(H)$ is the number  of  matrices $\cM_n\(\Z; H\)$ 
whose characteristic polynomial has  an irreducible factor of degree $r$, and 
 $Q_{n,r}(H)$ the number  of  matrices $\cM_n\(\Z; H\)$  with at least one multiple eigenvalue of degree at least $r+1$. The quantity $R_{n,r}(H)$ has been estimated in Lemma~\ref{lem:low deg eigenval}, and below we estimate $Q_{n,r}(H)$.

If $A\in \cM_n\(\Z; H\)$ has all eigenvalues of degree at least $r+1$ and at least one multiple eigenvalue 
then its characteristic polynomial $f \in \Z[X]$  is of the form $f = g^2h$ with monic polynomials 
\[
 g,h \in \Z[X] \mand \deg g \ge r+1.
 \]
We now choose  $p$ as the smallest prime with $p \ge 2H+1$. 
Clearly the reduction of $A$ modulo $p$ has characteristic polynomial 
 $f_p \in \F_p[X]$   of the form $f_p = g_p^2h_p$
 with monic polynomials 
\[
 g_p,h_p \in \Z[X] \mand  \deg g_p \ge r+1.
 \]
Thus there are 
\[
\sum_{r+1\le d \le n/2} p^{d} \cdot p^{n-2d} = \sum_{r+1\le d \le n/2} p^{n-d} \le 2 p^{n-r-1}
\]
possibilities of $f_p$. By  Lemma~\ref{lem:Rei} the number of $n \times n$ matrices over $\F_p$ 
with characteristic polynomials of this type is  
\begin{equation}
\label{eq:Count Fp}
O\(p^{n^2-n} \cdot p^{n-r-1}\) = O\(p^{n^2-r-1}\). 
\end{equation}
Since $p \ge 2H+1$, each of these  matrices over $\F_p$  can be lifted to at most one matrix 
$A\in \cM_n\(\Z; H\)$. Furthermore, since by our choice, $p = O(H)$, we see from~\eqref{eq:Count Fp}
that there are 
 \begin{equation}
\label{eq:High Deg}
Q_{n,r}(H)  =  O\(H^{n^2-r-1}\)
\end{equation}
such matrices $A\in \cM_n\(\Z; H\)$. 

Combining Lemma~\ref{lem:low deg eigenval} and~\eqref{eq:High Deg}, we derive
 \begin{equation}
\label{eq:Gen Bound}
R_n(H) \le R_{n,r}(H) + Q_{n,r}(H)  \ll H^{n^2-n+r(r+1)/2} \log H + H^{n^2-r-1}.
\end{equation}
Since $r\le n/2$ is arbitrary, the result follows. 


\subsection{The cases $n=2,5,7,8$} 
When $n=2$, then any matrix counted by $R_2(H)$ with eigenvalue $\lambda$  has the  characteristic polynomial of the form 
\[
(X-\lambda)^2.
\]
Since there are $O(H)$ possibilities for $\lambda$, applying Lemma~\ref{lem:char poly}, we conclude that for each possible eigenvalue $\lambda$ there are $H^{1+o(1)}$ matrices with characteristic polynomial $(X-\lambda)^2$. Therefore,
\[
R_2(H)\le H\cdot H^{1+o(1)}=H^{2+o(1)}.
\]

The improvement for the cases $n=5,7,8$ comes from the situation when $r$ satisfies the inequalities 
\begin{equation}
\label{eq:cond fact}
r+1>n/3 \mand (r+1)(r+2)<2n.
\end{equation}
This enforces $n<15$.

Indeed, following the proof in Section~\ref{sec:upper}, we note that, in the case when  all eigenvalues are of degree at least $r+1$ and at least one is a multiple eigenvalue, 
then its characteristic polynomial $f \in \Z[X]$  is of the form $f = g^2h$ with monic polynomials 
\[
 g,h \in \Z[X] \mand \deg g,\deg h \ge r+1.
 \]
Looking at the degree, one needs $n \ge 2(r+1)+r+1$, that is, $r+1\le n/3$. Therefore, if $r+1>n/3$ such a factorisation of $f$ cannot occur. Hence for such pairs $(n,r)$ we simply have 
\[
R_n(H) \le R_{n,r}(H)   \ll H^{n^2-n+r(r+1)/2} \log H .
\]

The second inequality in~\eqref{eq:cond fact} comes from imposing that the bound obtained on $R_{n,r}(H) $ 
in  Lemma~\ref{lem:low deg eigenval}
is smaller than the bound~\eqref{eq:High Deg} on  $Q_{n,r}(H) $ and hence the above 
bound is stronger than~\eqref{eq:Gen Bound}, 

Simple calculations show that this is possible for $n=5$ (with $r=1$), $n=7$ (with $r=2$) and $n=8$ (with $r=2$), 
and we get the desired upper bounds on $R_5(H)$, $R_7(H)$ and $R_8(H)$. 

\subsection{The lower bound} 
The lower bound comes from the following simple construction: we consider
$n\times n$ triangular matrices 
\[A=
\begin{pmatrix}
a & b & * & \cdots & *\\ 
0 & a & * & \cdots & *\\
0 & 0 & a_{3,3} & \cdots & *\\
\cdots & \cdots & \cdots & \cdots & \cdots\\
0 & 0 & 0 & \cdots & a_{n,n}
\end{pmatrix}, 
\]
with $a,b\in\Z$, $b\ne 0$, $|a|,|b|\le H$ and $a_{i,i}\ne a$, $|a_{i,i}|\le H$, $i=3,\ldots,n$.  

Clearly, the minor $M_{2,1}$ of 
\[
A-a\cdot I_n=
\begin{pmatrix}
0 & b & * & \cdots & *\\ 
0 & 0 & * & \cdots & *\\
0 & 0 & a_{3,3}-a & \cdots & *\\
\cdots & \cdots & \cdots & \cdots & \cdots\\
0 & 0 & 0 & \cdots & a_{n,n}-a
\end{pmatrix}, 
\]
(that is, the submatrix of  $A-a\cdot I_n$ obtained by removing the second row and the first column) is an upper triangular matrix with non-zero elements on the main diagonal, and hence it is non-singular. 
Therefore,  $A-a\cdot I_n$ is of rank $n-1$ and thus the geometric multiplicity of the eigenvalue $a$, satisfies  
\[
\dim \Ker(A- a\cdot I_n)=1.
\]
Since this is different from the algebraic multiplicity of $a$, which is $2$ (the multiplicity of $a$ as a root of the characteristic polynomial), the matrix $A$ is not diagonalisable.

Therefore, 
\[
R_n(H)\ge N_n(H)\gg H \cdot H^{n-2}\cdot H^{n(n-1)/2} = H^{(n^2+n-2)/2},
\] 
which concludes the proof.

   \section{Proof of Theorem~\ref{thm:Discr}}
   
 \subsection{The cases $n=2$}  
 The  characteristic polynomial of 
\[A=
\begin{pmatrix}
a_{11} & a_{12}\\ 
a_{21}& a_{22}
\end{pmatrix}, 
\]
is $X^2 - \(a_{11} + a_{22} \) X +  \(a_{11}  a_{22} - a_{12}  a_{21}\)$
and hence its discriminant is
\[
 \(a_{11} + a_{22} \)^2 - 4 \(a_{11}  a_{22} - a_{12}  a_{21}\)
 =    \(a_{11} - a_{22} \)^2 + 4a_{12}  a_{21}. 
 \]
 Considering the cases $\(a_{11} - a_{22} \)^2  = d$ and 
 $ \(a_{11} - a_{22} \)^2   \ne  d$
 separately, and using the classical bound on the divisor function, 
  see~\cite[Equation~(1.81)]{IwKow}, 
 we conclude the proof in this case.

   \subsection{Integer points on varieties}

We recall the following  general bound of Salberger~\cite[Theorems~0.4 and~7.4]{Salb}  (see also 
Browning, Heath-Brown and Salberger~\cite[Corollary~2 and Lemma~8]{BrHBSa}), 
which in turn improves the result of Pila~\cite[Theorem~A]{Pila} under 
some additional constraints.

\begin{lemma}\label{lem:BH-BS} Let $F(X_1, \ldots, X_m)\in \Z[X_1, \ldots, X_m]$ 
be   a polynomial of degree $ \deg F\ge 4$ in  $m\ge 3$ variables
 and with an absolutely irreducible highest form.  
 Then
\begin{align*}
\# \{(x_1, \ldots, x_m) \le  \(\Z\cap[-H,H]\)^m:~
F(x_1, \ldots, x_m)& = 0\}\\
& \le  H^{m-2+ o(1)},  
\end{align*}
as $H \to \infty$.
 \end{lemma}
 
%

      \subsection{Concluding the proof}
Let now $n \ge 3$.
 We observe that the characteristic polynomial of a matrix
  \[
\vX = \(X_{ij}\)_{1 \le i,j \le n} 
\]  
whose entries are considered as variables, is of the form
\[
f_\vX(Z) = Z^n + F_1 Z^{n-1} + \ldots +   F_{n-1} Z+ F_n, 
\]
where $F_i \in \Z[\vX]$ is a polynomial  in $n^2$ variables of total degree $i$, 
$i =1, \ldots, n$.  Hence,  $\disc \(f_\vX\)$ is also a polynomial in $X_{ij}$, $1 \le i,j \le n$.

We now investigate the degree of $\disc \(f_\vX\)$ as a polynomial 
 in $X_{ij}$, $1 \le i,j \le n$. Since $f_\vX$ is monic and square-free, the discriminant is given by
 $$
 \disc \(f_\vX\)=\prod_{i<j}(\lambda_j-\lambda_i)^2,
 $$
 where $\lambda_i$, $i=1,\ldots,n$, are the distinct eigenvalues. This polynomial is symmetric in $\lambda_1,\ldots,\lambda_n$, and thus can be expressed as a polynomial in the $n$ symmetric polynomials in $\lambda_1,\ldots,\lambda_n$, which in turn are homogeneous polynomials in the entries of $\vX$.  
Since $\disc \(a f_\vX\)=a^{n(n-1)}\disc \(f_\vX\)$ for any $a\in\C$, we conclude that $\disc \(f_\vX\)$ is a homogeneous polynomial with 
\begin{equation}
\label{eq: deg X}
 \deg \disc \(f_\vX\)=n(n-1).
\end{equation}


Since $n \ge 3$, by  Lemma~\ref{lem:abs irred discr} and also using~\eqref{eq: deg X}, we can apply  Lemma~\ref{lem:BH-BS} with $F$ replaced by $ \disc \(f_{\vX}\) - d$ and $m = n^2$, which instantly implies the desired result. 
 

  \section{Other applications of our approach}
 \label{sec:Appl}
  
  \subsection{Matrix lifting} We note that the  generalisation of the result from~\cite{Shp}, given in 
Lemma~\ref{lem:shift sing matr}, allows us to estimate then number of modular lifts of matrices 
with a prescribed determinant. Namely, given a positive integer $q$ and a matrix  
$B = \(b_{ij}\)_{1 \le i,j \le n} $ with $b_{ij} \in \{0, \ldots, q-1\}$, $1 \le i,j \le n$, we consider the set 
\[
\cT_n(d,H; q, B)=\{A\in  \cM_n(\Z; H) :~ A \equiv B \pmod q, \ \det A=d\}.
\]
For $d=1$, that is for $\SL_n(\Z)$ matrices, a related question about the smallest $H$ (as a function of $q$) for which 
$\cT_n(1,H; q, B) \ne \emptyset$  
has recently attracted a lot 
of attention. For example,  see~\cite[Section~1.3]{AsBl} for an upper bound on such $H$
in  the case of $d=1$ and almost all $B$ with  $B \equiv I_n \pmod q$ and~\cite{KaVa} for a lower bound on such $H$ 
in the worst case; see also references therein. This can also be asked
for other values of $d$. Here we remark that Lemma~\ref{lem:shift sing matr} 
applied with $K = - q^{-1} B$ and $\alpha = d/q^n$ (and with $\rf{H/q}$ instead
of $H$) immediately implies that for $H \ge 2q$ we have 
\begin{equation}
\label{eq:T-OS}
\# \cT_n(d,H; q, B) \ll (H/q)^{n^2-n} \log (H/q)
\end{equation}
uniformly over $d$ 
and $B$.   This can be compared with the 
 result of  Assing, Blomer and Nelson~\cite[Corollary~1.7]{AsBlNe}   
(see also~\cite[Theorem~1.4]{AsBl})  
which covers the case when  $d =1$ and  $B = I_n$  
is equivalent to the bound
\begin{equation}
\label{eq:T-AB}
\# \cT_n(1,H; q, I_n) \le \( \frac{ H^{n^2-n}}{q^{n^2-1}}  +H^{(n^2-n)/2}\) H^{o(1)}.
\end{equation}
The  bound~\eqref{eq:T-OS} is more general and stronger than~\eqref{eq:T-AB} for large $q$, namely,  for 
\[
q  \ge H^{1/2 + \varepsilon}
\]
for any fixed $\varepsilon > 0$.

 \subsection{Singular matrix polynomials} 
We now give   an applications of Lemma~\ref{lem:shift sing matr} to counting singular 
 matrix polynomials   $F(A)$ with a given $F \in \C[X]$ and $A \in   \cM_n(\Z; H)$. 
 More precisely, we have 
 \begin{equation}
\label{eq:f(A)-Sing}
\{A\in  \cM_n(\Z; H) :~\det F(A)=0\}\le H^{n^2-n+o(1)}.
\end{equation}
Indeed, the eigenvalues of $F(A)$ are of the shape $F(\lambda)$, 
where $\lambda$ runs through the eigenvalues of $A$. Hence, if 
$\det F(A)=0$ then we have $\det (A -\rho I_n) = 0$ where $\rho$ is one of the roots of $F$ 
and as before,  $I_n$ denotes the $n\times n$ identity matrix. 
Invoking Lemma~\ref{lem:shift sing matr}, we obtain~\eqref{eq:f(A)-Sing}.
  
   \subsection{Condition number of integral matrices} 
   
 We recall that the {\it condition number\/} of an $n\times n$ matrix $A$ (with arbitrary complex entries) 
 is given by 
\[
  \kappa(A) = \frac{\max_{i=1, \ldots, n} \sigma_i(A)}{\min_{i=1, \ldots, n} \sigma_i(A)}, 
\]
where $\sigma_1(A), \ldots, \sigma_n(A)$ are the singular values of $A$, that is,
\[
\sigma_i(A)= \sqrt{\lambda_i\(A A^T\)}, \]
where $\lambda_i\(A A^T\)$ are the eigenvalues of $A A^T$, $i=1, \ldots, n$,  
and $A^T$ denotes the transposition of $A$.  The condition number is responsible for the
numerical stability of many standard linear algebra algorithms (the smaller the better) 
and its distribution has been studied for a large variety of families of matrices, 
see~\cite{AnWe,CJMS, Dong, GoTi, RuVe, SUY, TaVu1, TaVu2, TaVu3} and references therein.
Here we address a discrete analogue of this question. For example, one can think about 
random matrices from $\cM_n(\Z; H)$ as obtained from random real matrices after 
decimal approximations of their entries, with a given precision, and then scaling the
new entries to make them integer. 

Since for $A\in  \cM_n(\Z; H)$,   the entries of $AA^T$ are of size  $O(H^2)$, by~\cite{Gersch} 
we conclude that 
\[
\max_{i=1, \ldots, n} \sigma_i(A) \ll H.
\]
This also implies that if  $\min_{i=1, \ldots, n} \sigma_i(A)\le  H/L $ then 
\[
|\det A| = \prod_{i=1}^n \sigma_i(A) \ll L^{-1}  H^{n} .
\]
Thus, we see from Lemma~\ref{lem:shift sing matr},  or~\cite[Theorem~4]{Shp}, 
that for all but $O\(L^{-1} H^{n^2}  \log H\)$ matrices  $A\in  \cM_n(\Z; H)$, 
we have $\kappa(A) \le L$. For example, taking $L= \log H \log \log H$, 
we see that   for almost all matrices  from  $\cM_n(\Z; H)$
the condition number is  reasonably small, namely, at most  $\log H \log \log H$.

\section*{Acknowledgement}
The authors would like to thank Wadim Zudilin  for useful comments and encouragement. The authors are grateful to the referee for the very careful reading and valuable comments, which improved some parts of the paper.

This work  was  supported, in part,  by the Australian Research Council Grants DP230100530 and DP230100534.

\end{document}